\documentclass[11pt,oneside,reqno]{amsart}

\usepackage{amssymb}
\usepackage{algpseudocode}
\usepackage{algorithm}
\usepackage{verbatim}
\usepackage{graphicx}
\usepackage[noadjust]{cite}
\RequirePackage{dsfont} 
 \scrollmode
\allowdisplaybreaks
\usepackage{graphicx}
\usepackage{epstopdf}
\usepackage[usenames]{color}
\definecolor{darkred}{RGB}{139,0,0}
\definecolor{darkgreen}{RGB}{0,100,0}
\definecolor{darkmagenta}{RGB}{139,0,139}

\usepackage{amsmath}
\usepackage{tikz}

\makeatletter
\newcommand{\xleftrightarrow}[2][]{\ext@arrow 3359\leftrightarrowfill@{#1}{#2}}
\makeatother

\newcommand{\xdasharrow}[2][->]{
\tikz[baseline=-\the\dimexpr\fontdimen22\textfont2\relax]{
\node[anchor=south,font=\scriptsize, inner ysep=1.5pt,outer xsep=2.2pt](x){#2};
\draw[shorten <=3.4pt,shorten >=3.4pt,dashed,#1](x.south west)--(x.south east);
}
}

\newcommand{\DEBUG}{}

\ifdefined\DEBUG%

  \def\rem#1{{\marginpar{\raggedright\scriptsize #1}}}
  \newcommand{\pmr}[1]{\rem{\color{blue}{$\bullet$ #1}}}
  \newcommand{\ppr}[1]{\rem{\color{red}{$\bullet$ #1}}}
 \else%

  \newcommand{\ppr}[1]{}
  \newcommand{\pmr}[1]{}
 \fi

\theoremstyle{plain}
\newtheorem{theorem}{Theorem}
\newtheorem{lemma}{Lemma}
\newtheorem{fact}{Fact}

\newtheorem{proposition}{Proposition}

\theoremstyle{definition}
\newtheorem{remark}{Remark}

\begin{document}

\title
[Randomized Euler schemes]
{On the randomized Euler schemes for ODEs \\ under inexact information}

\author[T. Bochacik]{Tomasz Bochacik}
\address{AGH University of Science and Technology,
Faculty of Applied Mathematics,
Al. A.~Mickiewicza 30, 30-059 Krak\'ow, Poland}
\email{bochacik@agh.edu.pl, corresponding author}

\author[P. Przyby{\l}owicz]{Pawe{\l} Przyby{\l}owicz}
\address{AGH University of Science and Technology,
Faculty of Applied Mathematics,
Al. A.~Mickiewicza 30, 30-059 Krak\'ow, Poland}
\email{pprzybyl@agh.edu.pl}

\begin{abstract}
We analyse errors of randomized explicit and implicit Euler schemes for approximate solving of ordinary differential equations (ODEs). We consider  classes of ODEs for which the right-hand side functions satisfy Lipschitz condition  globally or only locally. Moreover, we assume that only inexact discrete information, corrupted by some noise, about the right-hand side function is available. Optimality and stability of explicit and implicit randomized Euler algorithms are also investigated.
\end{abstract}
\maketitle
\tableofcontents
In this paper we consider ordinary differential equations (ODEs) of the following form: 
\begin{equation}
	\label{eq:ode}
		\left\{ \begin{array}{ll}
			z'(t)= f(t,z(t)), \ t\in [a,b], \\
			z(a) = \eta, 
		\end{array}\right.
\end{equation}
where $-\infty < a < b < \infty$, $\eta\in\mathbb{R}^d$, $f \colon [a,b]\times\mathbb{R}^d\to\mathbb{R}^d$, $d\in\mathbb{Z}_+$. We will consider the class of randomized algorithms and investigate the main properties of the randomized Euler schemes under inexact information, such as error bounds, optimality and stability.

Randomized algorithms for the approximate solving of ODEs have attracted attention in the recent years, see for example \cite{daun1, hein1, HeinMilla, JenNeuen, Kac1, KruseWu_1, stengle1, stengle2}. Nevertheless, there is still a~large space for further research on this topic in the setting of inexact information. This paper is an attempt in this direction. It is worth noting that related numerical problems -- such as function integration and approximation, approximate solving of PDEs, stochastic integration and SDEs -- are already extensively studied in the noisy information framework, see for instance  \cite{MoPl16, MoPl20, Wer96, Wer97, KaMoPr19, PMPP17, PMPP19}.

Randomized implicit and explicit Euler schemes have been investigated under exact information in the  articles \cite{backward_euler, JenNeuen, KruseWu_1, stengle1, stengle2}. In this paper we allow noisy information about the right-hand side function $f$. We will use similar assumptions as in \cite{randRK}, among which the key one is the (local or global) Lipschitz condition with respect to the state variable of $f$. The structure of this paper and considered problems resemble those from \cite{randRK} but here we consider the other algorithms; thus the current research can be viewed as a continuation of our previous work. 

We start from $L^p(\Omega)$-error analysis of randomized Euler schemes under inexact information. Result concerning error bound for the explicit scheme (Theorem \ref{theorem_explicit}) is more general, as it requires only local Lipschitz continuity of $f$ and linear growth of the noise function. Analogous result for the implicit scheme (Theorem \ref{theorem_implicit}) has been proven under global Lipschitz condition for $f$. Furthermore,  we establish lower error bounds for all algorithms based on randomized inexact  information in a certain class of right-hand side functions (Theorem \ref{lower_bounds}) and we provide condition for optimality of the randomized Euler schemes (Proposition \ref{fact_not_optimal}). Another novelty of our paper is stability analysis for the aforementioned algorithms. We characterize stability regions (mean-square, asymptotic and in probability) using a test problem designed to capture randomization in the time variable of $f$ and we show that stability regions are empty for the explicit scheme, whereas for the implicit scheme they cover almost entire complex plane, see Proposition \ref{prop_stab} and Remark \ref{rem_stab}. 

This paper is organized as follows. Section \ref{sec:prel} of this paper contains basic notation, assumptions about \eqref{eq:ode}, outline of the model of computation (including specification of the noise which corrupts values of the function $f$) and a definition of the $n$th minimal error. In sections \ref{sec:expl} and \ref{sec:impl} we establish upper bounds of the $L^p(\Omega)$-error of the randomized Euler schemes under inexact information (explicit and implicit, respectively). Lower bounds and condition for optimality of the randomized Euler schemes are discussed in section \ref{sec:low}. In section \ref{sec:stab} we propose a non-classical test problem to assess stability of the investigated algorithms and we show superiority of the implicit scheme in this aspect. Section 6 highlights main conclusions of the paper. Finally, in Appendix we gather some auxiliary results.

\section{Preliminaries} \label{sec:prel}
Let $\|\cdot\|$ be the one norm in $\mathbb{R}^d$, i.e. $\|x\|=\sum\limits_{k=1}^d|x_k|$ for $x\in \mathbb{R}^d$. For $x\in\mathbb{R}^d$ and $r\in [0,\infty)$ we denote by $B(x,r)=\{y\in\mathbb{R}^d \colon \|y-x\|\leq r\}$ the closed ball in $\mathbb{R}^d$ with center $x$ and radius $r$. Moreover, $B(x,\infty) = \mathbb{R}^d$ for all $x\in\mathbb{R}^d$. \smallskip

Let $(\Omega,\Sigma,\mathbb{P})$ be a complete probability space and let $\mathcal{N}=\{A\in\Sigma \colon \mathbb{P}(A)=0\}$. For a~random variable $X:\Omega\to\mathbb{R}$, defined on  $(\Omega,\Sigma,\mathbb{P})$, we denote its $L^p(\Omega)$ norm by $\|X\|_p=(\mathbb{E}|X|^p)^{1/p}$, $p\in [2,\infty)$. For a Polish space $E$ by $\mathcal{B}(E)$ we denote the Borel $\sigma$-field on $E$. \smallskip
Let $\varrho\in (0,1], K,L \in (0,\infty)$ and $R\in (0,\infty]$. As in \cite{randRK}, we consider a class $F^\varrho_R=F^\varrho_R(a,b,d,\varrho,K,L)$ of pairs $\left(\eta,f\right)$ satisfying the following conditions:
\begin{itemize}
\setlength\itemsep{4pt}
\item[(A0)] $\left\|\eta\right\|\leq K$,
\item[(A1)] $f\in \mathcal{C}\left([a,b]\times\mathbb{R}^d\right)$,
\item[(A2)] $\| f(t,x)\| \leq K\left(1+\left\|x\right\|\right)$ for all $(t,x)\in [a,b]\times\mathbb{R}^d$,
\item[(A3)] $\| f(t,x) - f(s,x) \|\leq L|t-s|^\varrho$ for all $t,s\in [a,b], x\in B\left(\eta,R\right)$,
\item[(A4)] $\| f(t,x) - f(t,y) \|\leq L\|x-y\|$ for all $t\in [a,b], x,y \in B\left(\eta,R\right)$.
\end{itemize}
Note that $F^{\varrho}_{\infty}$ consists of globally Lipschitz continuous functions, whereas functions from $F^{\varrho}_{0}$ may not satisfy Lipschitz condition even locally. Moreover, $F^{\varrho}_{\infty}\subset F^{\varrho}_{R}\subset F^{\varrho}_{R^{'}}\subset F^{\varrho}_{0}$ for $\infty\geq R\geq R^{'}\geq 0$.
Parameters of the class $F^\varrho_R$ are: $a, b, d, \varrho\in (0,1], K \in (0,\infty), L \in (0,\infty)$ and $R\in [0,\infty]$. These parameters, excluding $a$, $b$, and $d$, are usually not known in practical applications. Thus, they will be not used as an input of algorithms presented later in the paper. \smallskip

To approximate the solution of \eqref{eq:ode} for $f\in F^{\varrho}_{R}$, we will consider randomized algorithms based on inexact information about $f$. Now we will introduce the model of computation. Let us define the following two classes of noise functions:
\begin{equation}
\mathcal{K}_1(\delta) = \left\{ \tilde{\delta} \colon [a,b]\times\mathbb{R}^d\to\mathbb{R}^d \ \colon \ \tilde{\delta} \text{ is Borel measurable}, \|\tilde{\delta}(t,y)\|\leq\delta\left(1+\left\| y\right\|\right)  \hbox{for all} \ t\in [a,b], y\in \mathbb{R}^d \right\} \label{eq:K1_delta}
\end{equation}
and
\begin{equation}
\mathcal{K}_2(\delta) = \left\{ \tilde{\delta} \in \mathcal{K}_1(\delta) \colon \bigl\| \tilde{\delta}(t,x) - \tilde{\delta}(t,y) \bigr\| \leq \delta \|x-y\| \text{ for all }t\in[a,b], x,y\in\mathbb{R}^d \right\}, \label{eq:K_delta}
\end{equation}
where $\delta\in[0,1]$ is called the precision parameter. Note that $\mathcal{K}_2(\delta) \subset \mathcal{K}_1(\delta)$ for each $\delta \in [0,1]$ and there is no direct inclusion between these classes and the class $\mathcal{K}(\delta)$ considered in \cite{randRK}.

We assume that an algorithm may use only noisy evaluations of the function $f$. Specifically, for each point $(t,y)\in [a,b]\times\mathbb{R}^d$ we have
$$\tilde{f}(t,y) = f(t,y) + \tilde{\delta}_f(t,y),$$
where $\tilde \delta_f$ is an element of the class $\mathcal{K}_i(\delta)$ ($i\in\{1,2\}$ depending on which framework we choose) and $\tilde \delta_f(t, y)$ is an error corrupting the exact value $f(t,y)$. We allow randomized choice of the evaluation points $(t,y)$. Let 
\begin{equation*}
    V_{f}^i(\delta)=\{ \tilde f \colon \exists_{\tilde\delta_f\in\mathcal{K}_i(\delta)} \ \tilde f =f+\tilde\delta_f\}
\end{equation*} 
and 
\begin{equation*}
    V^i_{(\eta,f)}(\delta) = B(\eta,\delta)\times V_{f}^i(\delta)
\end{equation*}
for $(\eta,f)\in F^{\varrho}_{R}$, $\delta\in [0,1]$ and $i\in\{1,2\}$. Let us note that $V^i_{(\eta,f)}(\delta)\subset V^i_{(\eta,f)}(\delta')$ for $0\leq\delta\leq\delta'\leq 1$ and $V^i_{(\eta,f)}(0)=\{(\eta,f)\}$. Moreover, it holds that $V^2_{(\eta,f)}(\delta)\subset V^1_{(\eta,f)}(\delta)$. An additional regularity condition has been imposed on noise functions in class $\mathcal{K}_2(\delta)$ in order to establish convergence of the implicit Euler scheme under noisy information, cf. Theorem \ref{theorem_implicit}. \smallskip

Let $(\eta,f)\in F^{\varrho}_{R}$ and $(\tilde\eta,\tilde f)\in V^1_{(\eta,f)}(\delta)$. A vector of noisy information about $(\eta,f)$ takes the following form:
\begin{equation*}
    N(\tilde \eta,\tilde f)=[\tilde f(t_0,y_0),\ldots,\tilde f(t_{i-1},y_{i-1}),\tilde f(\theta_0,z_0),\ldots,\tilde f(\theta_{i-1},z_{i-1}),\tilde\eta],
\end{equation*}
where $i\in\mathbb{N}$ and $(\theta_0,\theta_1,\ldots,\theta_{i-1})$ is a random vector on  $(\Omega,\Sigma,\mathbb{P})$. Furthermore,
\begin{equation*}
    (y_0,z_0)=\psi_0(\tilde\eta),
\end{equation*}
and
\begin{equation*}
    (y_j,z_j)=\psi_j\Bigl(\tilde f(t_0,y_0),\ldots,\tilde f(t_{j-1},y_{j-1}),\tilde f(\theta_0,z_0),\ldots,\tilde f(\theta_{j-1},z_{j-1}),\tilde\eta\Bigr)
\end{equation*}
for Borel measurable mappings $\psi_j:\mathbb{R}^{(2j+1)d}\to\mathbb{R}^d\times \mathbb{R}^d$, $j\in\left\{0,\ldots,i-1\right\}$. In particular, this implies that $N(\tilde \eta,\tilde f):\Omega\to\mathbb{R}^{(2i+1)d}$ is a random vector. The total number of noisy evaluations of $f$ is $l=2i$.  \smallskip

We consider the class $\Phi$ of algorithms $\mathcal{A}$ which aim to compute the approximate solution $z$ of \eqref{eq:ode} using $N(\tilde\eta,\tilde f)$. Such algorithms have the following form:
\begin{equation*}
\label{def_alg}
    \mathcal{A}(\tilde\eta,\tilde f,\delta)=\varphi(N(\tilde \eta,\tilde f)),
\end{equation*}
where
\begin{equation*}
    \varphi:\mathbb{R}^{(2i+1)d}\to D([a,b];\mathbb{R}^d)
\end{equation*}
is a Borel measurable function -- in the Skorokhod space $D([a,b];\mathbb{R}^d)$, endowed with the Skorokhod topology, we consider the Borel $\sigma$-field $\mathcal{B}(D([a,b];\mathbb{R}^d))$. Therefore  $\mathcal{A}(\tilde\eta,\tilde f,\delta)\colon\Omega\to D([a,b];\mathbb{R}^d)$ is $\Sigma$-to-$\mathcal{B}(D([a,b];\mathbb{R}^d))$ measurable. Moreover, by Theorem 7.1 in \cite{Parth} the $\sigma$-field $\mathcal{B}(D([a,b];\mathbb{R}^d))$ coincides with the $\sigma$-field generated by coordinate mappings. Hence, for all $t\in [a,b]$ the mapping
\begin{equation}
\label{fix_t_alg}
    \Omega\ni \omega\mapsto \mathcal{A}(\tilde\eta,\tilde f,\delta)(\omega)(t)\in\mathbb{R}^d
\end{equation}
is $\Sigma$-to-$\mathcal{B}(\mathbb{R}^d)$-measurable. For a given $n\in\mathbb{N}$ we denote by $\Phi_n$ a class of all algorithms from $\Phi$ requiring at most $n$ noisy evaluations of $f$.

Let $p\in [2,\infty)$. For a fixed $(\eta,f)\in F^{\varrho}_{0}$ the error of $\mathcal{A}\in\Phi_n$ is given as
\begin{equation}
\label{fixed_err}
    e^{(p)}(\mathcal{A},\eta,f,V^i,\delta)=\sup\limits_{(\tilde \eta,\tilde f)\in V_{(\eta,f)}^i(\delta)} \Bigl\|\sup\limits_{a\leq t\leq b}\|z(\eta,f)(t)-\mathcal{A}(\tilde\eta,\tilde f,\delta)(t)\|\Bigl\|_p,
\end{equation}
for $i\in\{1,2\}$. (The error is well-defined, see \cite{randRK}, Remark 2.) The worst-case error of the algorithm $\mathcal{A}$ is
defined by
\begin{equation} \label{alg_err}
    e^{(p)}(\mathcal{A},\mathcal{G},V^i,\delta)=\sup\limits_{(\eta,f)\in \mathcal{G}} e^{(p)}(\mathcal{A},\eta,f,V^i,\delta),
\end{equation}
where $i\in\{1,2\}$ and $\mathcal{G}$ is a subclass of $F^{\varrho}_{0}$, see \cite{TWW88}. Finally,
we consider the $n$th minimal error defined as
\begin{equation}
\label{nth_min_err}
    e^{(p)}_n(\mathcal{G},V^i,\delta)=\inf\limits_{\mathcal{A}\in\Phi_n}e^{(p)}(\mathcal{A},\mathcal{G},V^i,\delta), \quad i\in\{1,2\}.
\end{equation}
Of course, we have that $e^{(p)}_n(\mathcal{G},V^2,\delta)\leq e^{(p)}_n(\mathcal{G},V^1,\delta)$.

The proposed framework of inexact information can be used in mathematical description of lowering precision, which is an important topic in the context of efficient computations on CPUs and GPUs. For details, see \cite{randRK, KaMoPr19, PMPP17, PMPP19}.

\section{Error analysis of the randomized explicit Euler scheme under inexact information} \label{sec:expl}

In this section we provide upper bound for $L^p(\Omega)$-error of the randomized explicit Euler scheme when the radius $R$ appearing in assumptions (A3) and (A4) is sufficiently large. The only parameters necessary to specify $R$ are $a,b,K$ (thus, $R$ does not depend on a particular IVP). 

The randomized explicit Euler method under inexact information is given by the following recurrence relation:
\begin{equation}\label{eq:explicit_euler}
\bar V^0 = \tilde{\eta},  \ \ \bar V^j = \bar V^{j-1} + h\cdot \tilde f\left(\theta_j, \bar V^{j-1}\right), \ j\in\{1,\ldots,n\},
\end{equation}
where $(\eta,f)\in F^\varrho_R$ for some $R> 0$, $(\tilde\eta,\tilde f)\in V^1_{(\eta,f)}(\delta)$, $n\in\mathbb{Z}_+$, $h=\frac{b-a}{n}$, $t_j = a+jh$ for $j\in\{0,1,\ldots,n\}$, $\theta_j = t_{j-1} + \tau_jh$ and $\tau_j\sim U(0,1)$ for $j\in\{1,\ldots,n\}$. We assume that $\{\tau_1,\ldots,\tau_n\}$ is an independent family of random variables. Note that $$\tilde f\left(\theta_j, \bar V^{j-1}\right) = f\left(\theta_j, \bar V^{j-1}\right) + \tilde{\delta}_f\left(\theta_j, \bar V^{j-1}\right)$$ for some $\tilde \delta_f\in\mathcal{K}_1(\delta).$
The solution of \eqref{eq:ode} is approximated by a piecewise linear function $\bar l^{EE} \colon [a,b]\to\mathbb{R}^d$ given by 
\begin{equation} \label{eq:le}
\bar l^{EE}(t)=\bar l^{EE}_j(t) \text{ for } t\in [t_{j-1},t_j], \ \ \bar l^{EE}_j(t) = \frac{\bar V^{j}-\bar V^{j-1}}{h}(t-t_{j-1})+\bar V^{j-1}, \ \ j\in\{1,\ldots,n\}.
\end{equation}
For $\delta=0$ (i.e. in case of exact information) we use notation without bars: $V^j$, $l^{EE}$ and $l^{EE}_j$ instead of $\bar V^j$, $\bar l^{EE}$ and $\bar l^{EE}_j$, respectively. 

The main result of this section (Theorem \ref{theorem_explicit}) will be preceded by two auxiliary facts. In Fact \ref{lemma:ball} we show that the sequence generated by the randomized explicit Euler scheme (under certain assumptions) falls inside the ball $B(\eta,R)$ for suitably chosen $R$. Note that this is the case also for the exact solution $t\mapsto z(t)$ of \eqref{eq:ode}, as stated in Lemma \ref{lm:sol}(i) in Appendix. Fact \ref{f:noise} in turn provides upper bound of the difference between sequences generated by the algorithm under exact ($\delta=0$) and inexact information. Hence,  we adapt the proof technique from  \cite{randRK} in order to cover the case considered in this paper.
\begin{fact} \label{lemma:ball}
Let 
\begin{equation} \label{eq:R1}
   R_1 = (K+2)e^{(K+1)(b-a)} + K-1.
\end{equation}
Then for all 
$(\eta,f)\in F^{\varrho}_{R_1}$, $\bigl(\tilde{\eta},\tilde{f}\bigr) \in V^1_{(\eta,f)}(\delta)$, $n\in\mathbb{Z}_+$, $\delta\in [0,1]$, and $j\in\{0,1,\ldots,n\}$
$$V^j,\bar V^j\in B(\eta,R_1)$$
almost surely.
\end{fact}
\begin{proof} 
Let us note that by considering any $\delta\in [0,1]$ we cover both cases $V^j$ and $\bar V^j$. By assumption (A0) and since $\tilde{\eta} \in B(\eta,\delta)\subset B(\eta,1)$,
$$\|\bar V^0\|\leq \|\eta\|+\|\tilde{\eta}-\eta\| \leq K+1.$$
Let $j\in\{1,\ldots,n\}$. Assumption (A2) and definition \eqref{eq:K1_delta} imply that the following inequality holds with probability $1$:
\begin{align*}
\bigl\|\bar V^j\bigr\| &  \leq  \bigl\| \bar V^{j-1}\bigr\| + h\bigl\| f(\theta_j,\bar V^{j-1})\bigr\|+h\bigl\|\tilde{\delta}_f(\theta_j, \bar V^{j-1} )\bigr\|  \\ &  \leq \bigl\| \bar V^{j-1}\bigr\|(1+h(K+1))+h(K+1).
\end{align*}
By discrete Gronwall's inequality:
\begin{equation}
    \bigl\|\bar V^j\bigr\| \leq \bigl\| \bar V^0\bigr\|(1+h(K+1))^n + (1+h(K+1))^n-1\leq  (K+2)e^{(K+1)(b-a)} -1 = R_1-K.\label{eq:V_bound}
\end{equation}
Note that \eqref{eq:V_bound} holds also for $j=0$ since $R_1-K \geq K+1 \geq\|\bar V^0\|$. By \eqref{eq:V_bound} and (A0),
$$\bigl\|\bar V^j-\eta \bigr\| \leq \bigl\|\bar V^j\bigr\|+\left\|\eta \right\| \leq R_1$$
for all $j\in\{0,1,\ldots,n\}$ and the proof is completed.
\end{proof}

\begin{fact} \label{f:noise}
Let $R_1$ be defined as in Fact \ref{lemma:ball}.
Then there exists a constant $C=C(a,b,K,L)>0$ such that for all $(\eta,f)\in F^{\varrho}_{R_1}$, $\bigl(\tilde{\eta},\tilde{f}\bigr) \in V^1_{(\eta,f)}(\delta)$, $n\in\mathbb{Z}_+$, $\delta \in [0,1]$ it holds
\begin{equation} \label{eq:4}
 \max_{0\leq j\leq n}\bigl\|V^j-\bar V^j\bigr\| \leq C\delta
\end{equation}
with probability $1$.
\end{fact}
\begin{proof}
Firstly, let us note that $ \| \bar V^0 - V^0\|  =  \|\tilde\eta-\eta\| \leq \delta$. For $j\in\{1,\ldots,n\}$, by (A4), \eqref{eq:K1_delta} and \eqref{eq:V_bound}, we obtain
\begin{align*}
    \bigl\| \bar V^j - V^j \bigr\| & \leq \bigl\| \bar V^{j-1} - V^{j-1} \bigr\| + h \bigl\| f(\theta_j, \bar V^{j-1})-f(\theta_j, V^{j-1}) \bigr\| + h \bigl\|\tilde \delta_f(\theta_j,\bar V^{j-1})\bigr\| \\ &  \leq (1+hL) \bigl\| \bar V^{j-1} - V^{j-1} \bigr\| + h\delta (1+R_1-K).
\end{align*}
By discrete Gronwall's inequality:
\begin{align*}
    \bigl\| \bar V^j - V^j \bigr\| & \leq \left(1+hL\right)^n \bigl\| \bar V^0-V^0\bigr\| + \frac{\delta (1+R_1-K)}{L}\cdot\bigl( (1+hL)^n-1 \bigr) \\ & \leq e^{L(b-a)}\left( 1 + \frac{1+R_1-K}{L} \right)\delta,
\end{align*}
which leads to \eqref{eq:4}.
\end{proof}

\begin{theorem} \label{theorem_explicit}
Let $p\in[2,\infty)$. There exists a constant $C = C(a,b,d,K,L,\varrho, p)>0$ such that for all $n\geq \lfloor b-a\rfloor+1$, $\delta\in [0,1]$, $(\eta,f)\in F^{\varrho}_{R_0}$, $\bigl(\tilde{\eta},\tilde{f}\bigr)\in V^1_{(\eta,f)}(\delta)$ it holds
$$\left\| \sup_{a\leq t \leq b} \left\| z(\eta,f)(t)-\bar l^{EE}(\tilde{\eta},\tilde{f},\delta)(t)\right\| \right\|_p \leq C\left(h^{\min\left\{\varrho+\frac12,1\right\}}+\delta\right),$$  where $R_0 = \max \left\{R_1, \ R_2 \right\}$, $R_1$ is given by \eqref{eq:R1} and $R_2$ -- by \eqref{eq:R2}.
\end{theorem}

\begin{proof}
Let $\Delta_j=[t_{j-1},t_j]$ and 
\begin{equation}\label{eq:z_bar}
    \bar z_j(t) = \frac{z(t_j)-z(t_{j-1})}{h}(t-t_{j-1}) + z(t_{j-1})
\end{equation}
for $\ t\in\Delta_j$, $\ j\in\{1,\ldots,n\}$. In this proof $C$ may denote different constants but always depending only on the following parameters: $a,b,d,K,L,\varrho, p$. 

Let us observe that
\begin{equation} \label{eq:1}
\Bigl\| \sup_{a\leq t \leq b} \left\| z(t)-\bar l^{EE}(t)\right\| \Bigr\|_p \leq \max_{1\leq j\leq n} \sup_{t\in\Delta_j}\|z(t)-\bar z_j(t)\| + \Bigl\| \max_{1\leq j\leq n}\sup_{t\in\Delta_j} \left\| \bar z_j(t)-\bar l^{EE}_j(t)\right\| \Bigr\|_p.
\end{equation}
We will find an upper bound for the first term in the right-hand side of \eqref{eq:1}. By the Lagrange mean value theorem for $t\in\Delta_j$ we get
$$z(t) = \left( z_1(t),\ldots,z_d(t) \right) = \left( z_1(t_{j-1})+z_1'(\alpha_{1j}^t)(t-t_{j-1}), \ldots, z_d(t_{j-1})+z_d'(\alpha_{dj}^t)(t-t_{j-1}) \right),$$
where $\alpha^t_{lj} \in [t_{j-1},t] \subset \Delta_j$, $j\in\{1,\ldots,n\}$, $l\in\{1,\ldots,d\}$. Moreover, $\bar z_j(t_j) = z(t_j)$ for $j\in\{0,1,\ldots,n\}$ and
$$\frac{z_l(t_j)-z_l(t_{j-1})}{h} = z'_l(\beta_{lj})$$
for some $\beta_{lj}\in (t_{j-1},t_j) \subset \Delta_j$, $j\in\{1,\ldots,n\}$, $l\in\{1,\ldots,d\}$. As a result, by \eqref{eq:5} in Lemma \ref{lm:sol}(ii), we obtain
\begin{align*}
\|z(t)-\bar z_j(t)\| & = \sum_{l=1}^d \left|z_l'(\alpha^t_{lj})(t-t_{j-1}) -z_l'(\beta_{lj})(t-t_{j-1}) \right|  \\ & \leq h\sum_{l=1}^d |z_l'(\alpha^t_{lj}) - z_l'(\beta_{lj}) |  \leq C h \sum_{l=1}^d \left| \alpha^t_{lj} - \beta_{lj} \right|^\varrho \\ & \leq Cd\cdot h^{\varrho+1}
\end{align*}
for all $t\in\Delta_j$, $j\in\{1,\ldots,n\}$, which leads to 
\begin{equation} \label{eq:6}
\max_{1\leq j\leq n} \sup_{t\in\Delta_j}\|z(t)-\bar z_j(t)\| \leq Ch^{\varrho+1}.
\end{equation}
Now we will analyze the second term in the right-hand side of \eqref{eq:1}.  
Let us observe that for each $\omega\in\Omega$ and $j\in\{1,\ldots,n\}$ there exist $\alpha_j,\beta_j\in \mathbb{R}^d$ such that $\bar z_j(t)-\bar l^{EE}_j(t)=\alpha_j t+\beta_j$ for all $t\in\Delta_j$. Moreover, for each $r>0$ such that $\|\alpha_j t_{j-1} +\beta_j \|< r$ and $\|\alpha_j t_{j} +\beta_j \|< r$ the following holds: $$\|\alpha_j t +\beta_j \| \leq \frac{t_j-t}{t_j-t_{j-1}}\|\alpha_j t_{j-1} +\beta_j \| + \left(1-\frac{t_j-t}{t_j-t_{j-1}}\right)\|\alpha_j t_{j} +\beta_j \| < r$$ for all $t\in\Delta_j$. Taking $r=\max\left\{ \|z(t_{j-1})-\bar V^{j-1}\|,\| z(t_j)-\bar V^j \|\right\} $ leads to $$ \sup_{t\in\Delta_j} \bigl\| \bar z_j(t)-\bar l^{EE}_j(t)\bigr\| = \max\left\{ \|z(t_{j-1})-\bar V^{j-1}\|,\|z(t_j)-\bar V^j\|\right\}$$ and as a result
\begin{equation} \label{eq:2}
\Bigl\| \max_{1\leq j\leq n}\sup_{t\in\Delta_j} \bigl\| \bar z_j(t)-\bar l^{EE}_j(t)\bigr\| \Bigr\|_p \leq \Bigl\| \max_{0\leq j\leq n}\|z(t_j)-V^j\|\Bigr\|_p+\Bigl\| \max_{0\leq j\leq n}\|V^j-\bar V^j\|\Bigr\|_p .
\end{equation}
For $k\in\{1,\ldots,n\}$:
\begin{align}
    z\left(t_k\right)-V^k &  = \sum_{j=1}^k \left( z\left(t_j\right)-z\left(t_{j-1}\right) \right) - \sum_{j=1}^k \left( V^j - V^{j-1}\right) \notag \\ & = \sum_{j=1}^k \int\limits_{t_{j-1}}^{t_j}z'(s)\,\mathrm{d}s - h\sum_{j=1}^k f\left( \theta_j, V^{j-1} \right) \notag \\ & = S_1^k + S_2^k+S_3^k, \label{eq:S1S2S3}
\end{align}
where 
\begin{align*}
    S_1^k & = \sum_{j=1}^k\Bigl( \int\limits_{t_{j-1}}^{t_j} z'(s)\,\mathrm{d}s - hz'(\theta_j)\Bigr), \\ S_2^k & = h \sum_{j=1}^k\left(f(\theta_j,z(\theta_j))-f(\theta_j,z(t_{j-1})) \right), \\ S_3^k & = h \sum_{j=1}^k\left(f(\theta_j,z(t_{j-1}))-f(\theta_j,V^{j-1}) \right).
\end{align*}
Now we will show that 
\begin{equation} \label{eq:S3}
 \Bigl\| \max_{0\leq j\leq n} \|z(t_j)-V^j\| \Bigr\|_p \leq e^{L(b-a)}\cdot \left( \Bigl\| \max_{1\leq k\leq n} \|S_1^k\|  \Bigr\|_p + \Bigl\| \max_{1\leq k\leq n} \|S_2^k\|  \Bigr\|_p \right).
\end{equation}
Let us define $u_0=0$ and $$u_k = \Bigl\| \max_{0\leq j\leq k} \|z(t_j)-V^j\|\Bigr\|_p= \Bigl\| \max_{1\leq j\leq k} \|z(t_j)-V^j\|\Bigr\|_p,$$ for $k\in\{1,\ldots,n\}$. Then by (A4) we get
\begin{align*}
\Bigl\|\max_{1\leq j\leq k}\|S_3^j\|\Bigr\|_p & \leq hL \sum_{j=1}^k \Bigl\| \left\|z(t_{j-1})-V^{j-1} \right\| \Bigr\|_p  \leq hL \sum_{j=0}^{k-1} u_j.
\end{align*}
From the above and \eqref{eq:S1S2S3} we obtain
\begin{align*}
u_k  \leq \Bigl\| \max_{1\leq j\leq n} \|S_1^j\|\Bigr\|_p+ \Bigl\|\max_{1\leq j\leq n}\|S_2^j\|\Bigr\|_p+hL \sum_{j=0}^{k-1} u_j.
\end{align*}
Inequality \eqref{eq:S3} follows from discrete Gronwall's inequality.

Since $ h\sum\limits_{j=1}^k z'(\theta_j)$ is the randomized Riemann sum of $ \int\limits_{t_0}^{t_j} z'(s)\,\mathrm{d}s$ and $z'$ is $\varrho$-H\"{o}lder continuous (as stated in Lemma \ref{lm:sol}(ii)), we can use Theorem 3.1 from \cite{KruseWu_1} to show that 
\begin{equation} \label{eq:S1}
\Bigl\|\max_{1\leq k\leq n} \|S_1^k\|\Bigr\|_p \leq Ch^{\varrho+\frac12}.
\end{equation}
In fact, by (3.3) in Theorem 3.1:
$$ \Bigl\|\max_{1\leq k\leq n} \|S_1^k\|\Bigr\|_p \leq \sqrt{d} \cdot\Bigl\|\max_{1\leq k\leq n} \|S_1^k\|_2\Bigr\|_p \leq Ch^{\varrho+\frac12}\cdot \|z'\|_{\mathcal{C}^\varrho([a,b])}, $$
where $\|\cdot\|_2$ is the Euclidean norm in $\mathbb{R}^d$ and $\|g\|_{\mathcal{C}^\varrho([a,b])}$
is the H\"{o}lder norm for each $\varrho$-H\"{o}lder continuous function $g$.

Furthermore, 
\begin{equation} \label{eq:S2}
\Bigl\|\max_{1\leq k\leq n} \|S_2^k\|\Bigr\|_p \leq Ch
\end{equation}
because by (A4) and \eqref{eq:diff_z} in Lemma \ref{lm:sol}(ii) we have
\begin{align*}
\max_{1\leq k\leq n}\bigl\| S_2^k\bigr\| & \leq h\sum_{j=1}^n\left\|f(\theta_j,z(\theta_j))-f(\theta_j,z(t_{j-1})) \right\| \leq hL\sum_{j=1}^n C \left|\theta_j-t_{j-1}\right| \leq  hLC(b-a).
\end{align*}

From \eqref{eq:S3}, \eqref{eq:S1} and \eqref{eq:S2} it follows that
\begin{align}
\Bigl\| \max_{0\leq j\leq n} \|z(t_j)-V^j\| \Bigr\|_p & \leq C\left(h^{\varrho+\frac12}+h\right) \leq C h^{\min\left\{\varrho+\frac12,1\right\}} \cdot \left(1+(b-a)^{\left|\varrho-\frac12\right|}\right).\label{eq:final1}
\end{align}
From \eqref{eq:2}, \eqref{eq:final1} and Fact \ref{f:noise} we get
\begin{equation} \label{eq:final2}
\Bigl\| \max_{1\leq j\leq n}\sup_{t\in\Delta_j} \bigl\| \bar z_j(t)- \bar l^{EE}_j(t)\bigr\| \Bigr\|_p \leq C\left( h^{\min\left\{\varrho+\frac12,1\right\}}+\delta \right).
\end{equation}
By \eqref{eq:1}, \eqref{eq:6} and \eqref{eq:final2}, we obtain the desired claim.
\end{proof}

\section{Error analysis of the randomized implicit Euler scheme under inexact information} \label{sec:impl}
In this section we consider the class $F^{\varrho}_{\infty}$ of IVPs having the form \eqref{eq:ode} for which the right-hand side function satisfies the global Lipschitz condition.

The randomized implicit Euler method under inexact information is defined as follows. Let $n\in\mathbb{Z}_+$, $h=\frac{b-a}{n}$, $t_j = a+jh$ for $j\in\{0,1,\ldots,n\}$ and $\theta_j = t_{j-1} + \tau_jh$ for $j\in\{1,\ldots,n\}$, where $\tau_j\sim U(0,1)$ for $j\in\{1,\ldots,n\}$ are independent random variables on $(\Omega,\Sigma,\mathbb{P})$. Let $(\eta,f)\in F^{\varrho}_{\infty}$ and $(\tilde\eta,\tilde f)\in V^2_{(\eta,f)}(\delta)$. Iterations of the algorithm are given as follows: 
\begin{equation}
\label{eq:implicit_euler}
\bar U^0 = \tilde \eta,  \ \ \bar U^j = \bar U^{j-1} + h\cdot \tilde f\bigl(\theta_j, \bar U^{j}\bigr), \ j\in\{1,\ldots,n\}.
\end{equation}
Note that
$$\tilde f\bigl(\theta_j, \bar U^{j}\bigr) = f\bigl(\theta_j, \bar U^{j}\bigr) + \tilde \delta_f\bigl(\theta_j, \bar U^{j}\bigr)$$
for some $\tilde \delta_f\in \mathcal{K}_2(\delta)$. The solution to \eqref{eq:ode} is approximated by $\bar l^{IE} \colon [a,b]\to\mathbb{R}^d$ given by 
$$\bar l^{IE}(t)=\bar l^{IE}_j(t) \text{ for } t\in [t_{j-1},t_j], \ \ \bar l^{IE}_j(t) = \frac{\bar U^{j}-\bar U^{j-1}}{h}(t-t_{j-1})+\bar U^{j-1}, \ \ j\in\{1,\ldots,n\}.$$
By $U^j$, $l^{IE}$ and $l^{IE}_j$ we denote counterparts of $\bar U^j$, $\bar l^{IE}$ and $\bar l^{IE}_j$ under exact information.

The first step in our analysis is to show that for sufficiently small $h$ the algorithm has the solution, i.e. at each iteration there exists $\bar U^j$ satisfying \eqref{eq:implicit_euler}. To prove this fact we introduce the following filtration: $\mathcal{F}_0=\sigma(\mathcal{N})$ and $\mathcal{F}_j =\sigma\bigl( \sigma\left(\tau_1,\ldots,\tau_j\right)\cup\mathcal{N}\bigr)$ for $j\in\{1,\ldots,n\}$.
\begin{lemma} \label{lm:exist}
Let $\delta\in[0,1]$, $(\eta,f)\in F^{\varrho}_{\infty}$, $\bigl(\tilde{\eta},\tilde{f}\bigr) \in V^2_{(\eta,f)}(\delta)$, $n\in\mathbb{Z}_+$, $h(K+1)\leq\frac12$ and $h(L+1)<1$. Then there exists a unique solution $\left(\bar U^j\right)_{j=0}^n$ to the randomized implicit Euler scheme \eqref{eq:implicit_euler} under inexact information such that $\sigma(\bar U^j)\subset\mathcal{F}_j$ for $j\in\{0,1,\ldots,n\}$ and
\begin{equation} \label{eq:U_bound}
    \max\limits_{0\leq j\leq n}\|\bar U^j\|\leq (K+2)e^{2(K+1)(b-a)}-1
\end{equation}
with probability $1$.
\end{lemma}
\begin{proof}[Proof] 
We will proceed by induction. Of course  $\bar U^0=\tilde{\eta}$ is deterministic and hence  $\mathcal{F}_0$-measurable. Let us assume that there exists  $\mathcal{F}_{j-1}$-measurable solution $\bar U^{j-1}$ to \eqref{eq:implicit_euler} for some $j\in\{1,\ldots,n\}$. 

We define a mapping $p_j \colon \Omega\times\mathbb{R}^d \to \mathbb{R}^d$ by the following formula:
\begin{equation}
p_j\left(\omega,x\right) = \bar U^{j-1}(\omega) + h \tilde f\left( \theta_j(\omega), x\right), \quad  (\omega,x)\in\Omega\times\mathbb{R}^d    
\end{equation}
and a mapping $h_j \colon \Omega\times\mathbb{R}^d \to \mathbb{R}^d$ by taking
\begin{equation}
     h_j(\omega,x) = p_j(\omega,x) - x, \quad  (\omega,x)\in\Omega\times\mathbb{R}^d.
\end{equation}
For each $\omega\in\Omega$ the function $x\mapsto h_j(\omega,x)$ is continuous, while for every $x\in\mathbb{R}^d$ the function $\omega\mapsto  h_j(\omega,x)$ is $\mathcal{F}_j$-measurable since $\sigma(\bar U^{j-1})\cup\sigma(\theta_j)\subset\mathcal{F}_j$. 
By (A4) and \eqref{eq:K_delta} we have for all $\omega\in\Omega$, $x,y\in\mathbb{R}^d$ that 
 \begin{align*}
\left\| p_j(\omega,x)-p_j(\omega,y)\right\| &  \leq h \bigl\| f(\theta_j(\omega),x)-f(\theta_j(\omega),y)\bigr\| +h \bigl\| \tilde \delta_f(\theta_j(\omega),x)-\tilde \delta_f(\theta_j(\omega),y)\bigr\|  \\ & \leq h (L+1)\left\| x-y\right\|. 
\end{align*}
Since $h(L+1)<1$, the function $x\mapsto p_j(\omega,x)$ is a contraction mapping for every $\omega\in\Omega$. From the Banach fixed-point theorem it follows that for every $\omega\in\Omega$ there exists a unique root $\bar U^j(\omega)\in\mathbb{R}^d$ of the function $x\mapsto h_j(\omega,x)$.
Hence, according to Lemma \ref{lemma:measurable} the mapping $\omega\mapsto\bar U^j(\omega)$ is $\mathcal{F}_j$-measurable and such that $\bar U^j = \bar U^{j-1} + h\cdot \tilde f\bigl(\theta_j, \bar U^{j}\bigr)$ with probability $1$.

Let us note that $\|\bar U^0\| \leq \|\eta\| + \delta \leq K+1$.
Moreover, from (A4), \eqref{eq:K1_delta} and \eqref{eq:K_delta}, we obtain
\begin{align*}
    \|\bar U^j\|  \leq \| \bar U^{j-1}\| + h\Bigl(\| f (\theta_j,\bar U^j) \| + \|\tilde \delta_f(\theta_j,\bar U^j) \| \Bigr) \leq \| \bar U^{j-1}\| + h(K+1 )(1+\| \bar U^j \|).
\end{align*}
for $j\in\{1,\ldots,n\}$. Since $$0<\frac{1}{1-h(K+1)}\leq 1+2h(K+1) \leq 2$$ for $0<h(K+1)\leq \frac12$, we obtain \begin{align*}
     \|\bar U^j\|  & \leq \frac{1}{1-h(K+1)}\|\bar U^{j-1}\|  +\frac{h(K+1)}{1-h(K+1)} \\ & \leq 
     \bigl( 1 + 2h(K+1)\bigr) \|\bar U^{j-1}\| +2h(K+1).
\end{align*}
By discrete Gronwall's inequality we obtain for $j\in\{0,1,\ldots,n\}$ the following inequality:
\begin{equation*}
     \|\bar U^j\|  \leq  \bigl( 1 + 2h(K+1)\bigr)^n \|\bar U^{0}\| + \bigl( 1 + 2h(K+1)\bigr)^n-1  \leq (K+2)e^{2(K+1)(b-a)}-1.
\end{equation*}
This concludes the proof.
\end{proof}

Now we will establish a result similar to Fact \ref{f:noise}.

\begin{fact} \label{lemma_U}
There exists a constant $C=C(a,b,K,L)>0$ such that
$$\max_{0\leq j\leq n} \bigl\| U^j-\bar U^j\bigr\| \leq C \delta $$
with probability $1$ for all $\delta \in [0,1]$, $(\eta,f)\in F^{\varrho}_{\infty}$, $\bigl(\tilde{\eta},\tilde{f}\bigr) \in V^2_{(\eta,f)}(\delta)$ and $n\in\mathbb{Z}_+$ such that $h(K+1) \leq \frac12$ and $hL\leq\frac12$.
\end{fact}
\begin{proof}[Proof]
Let us note that $
    \| \bar U^0 - U^0\|  =  \|\tilde\eta-\eta\| \leq \delta$.
For $j\in\{1,\ldots,n\}$ we obtain
\begin{align*}
    \| \bar U^j - U^j \| &  \leq \| \bar U^{j-1} - U^{j-1} \| + h \| f(\theta_j, \bar U^j)-f(\theta_j, U^j) \| + h \|\tilde \delta_f(\theta_j,\bar U^j)\| 
    \\ & \leq \| \bar U^{j-1} - U^{j-1} \| + hL \|  \bar U^j- U^j \| + h\delta (1+C),
\end{align*}
where $C$ is a bound for $\|\bar U^j\|$ given by \eqref{eq:U_bound}.
Thus,
\begin{align*}
    \| \bar U^j - U^j \| & \leq \frac{1}{(1-hL)} \| \bar U^{j-1} - U^{j-1} \| + \frac{h (1+C)}{1-hL}\delta \\ & \leq (1+2hL) \| \bar U^{j-1} - U^{j-1} \| + 2h (1+C)\delta
\end{align*}
and by discrete Gronwall's inequality
\begin{align*}
    \| \bar U^j - U^j \| & \leq  (1+2hL)^n \| \bar U^{0} - U^{0} \| + \frac{(1+C)\delta}{L}\cdot\bigl((1+2hL)^n-1\bigr) \\ & \leq e^{2L(b-a)}\cdot \delta + \frac{1+C}{L}\cdot\bigl(e^{2L(b-a)}-1\bigr)\cdot\delta,
\end{align*}
which completes the proof.
\end{proof}

The following theorem is an analogue of Theorem \ref{theorem_explicit} for the implicit version of randomized Euler scheme. We will use similar error decomposition as in \eqref{eq:S1S2S3} but with an extra term representing a shift from $f\left( \theta_j, U^{j-1}\right)$ to $f\left( \theta_j, U^{j}\right)$.

\begin{theorem} \label{theorem_implicit}
Let $p\in[2,\infty)$. There exists a constant $C = C(a,b,d,K,L,\varrho, p)>0$ such that $$\left\| \sup_{a\leq t \leq b} \left\| z(\eta,f)(t)-\bar l^{IE}(\tilde{\eta},\tilde{f},\delta)(t)\right\| \right\|_p \leq C\left(h^{\min\left\{\varrho+\frac12,1\right\}}+\delta\right)$$ for all $\delta\in [0,1]$, $(\eta,f)\in F^{\varrho}_{\infty}$, $\bigl(\tilde{\eta},\tilde{f}\bigr)\in V^2_{(\eta,f)}(\delta)$ and $n\in\mathbb{Z}_+$ such that $h(K+1) \leq \frac12$ and $hL\leq\frac12$. 
\end{theorem}
\begin{proof}
Proceeding analogously as in the proof of Theorem \ref{theorem_explicit}, we obtain 
\begin{equation*}
    \Bigl\| \sup_{a\leq t \leq b} \bigl\| z(t)-\bar l^{IE}(t)\bigr\| \Bigr\|_p \leq Ch^{\varrho+1} + \Bigl\| \max_{0\leq j\leq n}\|z(t_j)-U^j\|\Bigr\|_p+\Bigl\| \max_{0\leq j\leq n}\|U^j-\bar U^j\|\Bigr\|_p
\end{equation*}
-- cf. formulas \eqref{eq:1}, \eqref{eq:6} and \eqref{eq:2}. By Fact \ref{lemma_U}, 
\begin{equation} \label{eq:0i}
    \Bigl\| \sup_{a\leq t \leq b} \left\| z(t)-\bar l^{IE}(t)\right\| \Bigr\|_p \leq C\left( h^{\varrho+1} + \delta\right) + \Bigl\| \max_{0\leq j\leq n}\|z(t_j)-U^j\|\Bigr\|_p
\end{equation}
For $k\in\{1,\ldots,n\}$ it holds that 
\begin{align*}
    z\left(t_k\right)-U^k & = 
    S_1^k + S_2^k+S_3^k+S_4^k,
\end{align*}
where 
\begin{align*}
    S_1^k & = \sum_{j=1}^k \Bigl( \int\limits_{t_{j-1}}^{t_j} z'(s)\,\mathrm{d}s - hz'\left(\theta_j\right) \Bigr), \\
    S_2^k & = h \sum_{j=1}^k \left[ z'\left(\theta_j\right) - f\left( \theta_j, z\left(t_{j-1}\right) \right) \right], \\
    S_3^k & = h \sum_{j=1}^k \left[ f\left( \theta_j, z\left(t_{j-1}\right)\right) - f\left( \theta_j, U^{j-1}\right) \right], \\
    S_4^k & = h \sum_{j=1}^k \left[ f\left( \theta_j, U^{j-1}\right) - f\left( \theta_j, U^{j}\right) \right].
\end{align*}

By similar arguments as in the proof of Theorem \ref{theorem_explicit}, we show that 
\begin{equation*}
 \Bigl\| \max_{0\leq j\leq n} \|z(t_j)-U^j\| \Bigr\|_p \leq e^{L(b-a)}\cdot \left( \Bigl\| \max_{1\leq j \leq n} \|S_1^j\|  \Bigr\|_p + \Bigl\| \max_{1\leq j \leq n} \|S_2^j\|  \Bigr\|_p+\Bigl\| \max_{1\leq j\leq n} \|S_4^j\|  \Bigr\|_p \right),
\end{equation*}
cf. \eqref{eq:S3}. Furthermore, by analogy to \eqref{eq:S1} and \eqref{eq:S2},
\begin{equation} \label{eq:1i}
 \Bigl\| \max_{0\leq j\leq n} \|z(t_j)-U^j\| \Bigr\|_p \leq C \cdot \left( h^{\min \left\{ 1,\varrho+\frac12 \right\}}+\Bigl\| \max_{1\leq j\leq n} \|S_4^j\|  \Bigr\|_p \right).
\end{equation}
Using assumptions (A2) and (A4), we obtain inequality
\begin{align*}
    \bigl\| S_4^k\bigr\| & \leq h\sum_{j=1}^k \left\| f\left( \theta_j, U^{j-1}\right) - f\left( \theta_j, U^{j}\right) \right\|  \leq h^2L\sum_{j=1}^k \left\| U^{j-1}-U^j\right\| \\ & =  h^2L\sum_{j=1}^k \left\| f\left( \theta_j, U^j \right)\right\| \leq h^2LK \sum_{j=1}^k \left( 1+\left\|U^j\right\| \right)  \\ &  \leq hLK(b-a) \cdot \left( 1+ C \right)
\end{align*}
for all $k\in\{1,\ldots,n\}$ with probability $1$, where $C$ is a bound from \eqref{eq:U_bound}. This combined with \eqref{eq:1i} leads to 
\begin{equation} \label{eq:2i}
 \Bigl\| \max_{0\leq j\leq n} \|z(t_j)-U^j\| \Bigr\|_p \leq C h^{\min \left\{ 1,\varrho+\frac12 \right\}}.
\end{equation}
The thesis follows from \eqref{eq:0i} and \eqref{eq:2i}.
\end{proof}

\section{Lower bounds and optimality of randomized Euler schemes} \label{sec:low}
In this section we prove the following main result of the paper.
\begin{theorem}\label{lower_bounds}
Let $p\in [2,+\infty)$, $\varrho\in (0,1]$. There exist $C_1,C_2>0$, $n_0\in \mathbb{N}$, $\delta_0\in [0,1]$ such that for all $n\geq n_0$, $\delta\leq\delta_0$ the following holds
\begin{equation} \label{eq:lower}
    C_1(n^{-(\varrho+\frac{1}{2})}+\delta)\leq e^{(p)}_n(F^{\varrho}_{\infty},V^2,\delta)\leq e^{(p)}_n(F^{\varrho}_{R_0},V^1,\delta)\leq C_2 (n^{-\min\{1,\varrho+\frac{1}{2}\}}+\delta),
\end{equation}
where $R_0$ is defined in Fact \ref{f:noise}.
\end{theorem}
\begin{proof}
To show the first inequality let us note that 
\begin{equation} \label{lb_1}
     e^{(p)}_n(F^{\varrho}_{\infty},V^2,\delta) \geq e^{(p)}_n(F^{\varrho}_{\infty},V^2,0)  = \Omega(n^{-(\varrho+1/2)}), 
\end{equation}
when $\ n \to \infty$. This follows from lower bounds for an integration problem of H\"older continuous functions, see \cite{HeinMilla} and \cite{NOV} for the details. Furthermore, for any algorithm $\mathcal{A}\in\Phi_n$ and any $(\eta_1,f_1),(\eta_2,f_2)\in F^{\varrho}_\infty$ such that $V^2_{(\eta_1,f_1)}(\delta)\cap V^2_{(\eta_2,f_2)}(\delta)\neq\emptyset$ we have
    \begin{equation*} 
         e^{(p)}(\mathcal{A},F^{\varrho}_{\infty},V^2,\delta)\geq \frac{1}{2}\sup\limits_{a\leq t\leq b}\|z(\eta_1,f_1)(t)-z(\eta_2,f_2)(t)\|
    \end{equation*}
because for any $(\tilde\eta,\tilde f)\in V^2_{(\eta_1,f_1)}(\delta)\cap V^2_{(\eta_2,f_2)}(\delta)$ and for any $\mathcal{A}\in\Phi_n$ the following holds:
\begin{align*}
\sup\limits_{a\leq t\leq b}& \|z(\eta_1,f_1)(t)-z(\eta_2,f_2)(t)\| \\ & \leq \Bigl\|\sup\limits_{a\leq t\leq b}\|z(\eta_1,f_1)(t)-\mathcal{A}(\tilde\eta,\tilde f,\delta)(t)\| \Bigr\|_p+ \Bigl\|\sup\limits_{a\leq t\leq b}\|z(\eta_2,f_2)(t)-\mathcal{A}(\tilde\eta,\tilde f,\delta)(t)\|\Bigr\|_p \\ & \leq e^{(p)}(\mathcal{A},\eta_1,f_1,V^2,\delta)+e^{(p)}(\mathcal{A},\eta_2,f_2,V^2,\delta) \leq 2e^{(p)}(\mathcal{A},F^{\varrho}_{\infty},V^2,\delta).
\end{align*}   
Let us take $(\eta_1,f_1)=(0e_1,\delta e_1)$, $(\eta_2,f_2)=(0e_1,-\delta e_1)$, where $e_1 = (1,0,\ldots,0)$. These paris belong to $F^{\varrho}_\infty$ if $\delta\leq\min\{K,1\}$. Then $(0,0)\in V^2_{(\eta_1,f_1)}(\delta)\cap V^2_{(\eta_2,f_2)}(\delta)$, $z(\eta_1,f_1)(t)=\delta (t-a) e_1$ and $z(\eta_1,f_2)(t)=-\delta (t-a) e_1$. Thus,
    \begin{equation}
    \label{lb_2}
        e^{(p)}(\mathcal{A},F^{\varrho}_{\infty},V^2,\delta)\geq 
          \frac12 \sup\limits_{a\leq t\leq b}\| 2\delta (t-a)e_1 \| = (b-a)\delta.
    \end{equation}
By \eqref{lb_1} and \eqref{lb_2} we obtain the first inequality in \eqref{eq:lower}. The second inequality follows from the fact that $F^\varrho_\infty\subset F^\varrho_{R_0}$ and $V^2_{(\eta,f)}(\delta)\subset V^1_{(\eta,f)}(\delta)$. Indeed, 
\begin{align*} 
e^{(p)}_n(F^{\varrho}_{\infty},V^2,\delta) & = \inf_{\mathcal{A}\in \Phi_n} \sup_{(\eta,f)\in F^\varrho_\infty} e^{(p)}(\mathcal{A},\eta,f,V^2,\delta) \leq \inf_{\mathcal{A}\in \Phi_n} \sup_{(\eta,f)\in F^\varrho_{R_0}} e^{(p)}(\mathcal{A},\eta,f,V^2,\delta) \\ & \leq \inf_{\mathcal{A}\in \Phi_n} \sup_{(\eta,f)\in F^\varrho_{R_0}} e^{(p)}(\mathcal{A},\eta,f,V^1,\delta) = e^{(p)}_n(F^{\varrho}_{R_0},V^1,\delta).
\end{align*} 
Finally, the last inequality in \eqref{eq:lower} is a consequence of Theorem \ref{theorem_explicit}.
\end{proof}

The result above implies that when $\varrho\in (0,1/2]$, both randomized Euler schemes are optimal -- implicit version in the class of globally Lipschitz right-hand side functions $F^{\varrho}_{\infty}$, whereas explicit version in a broader class $F^{\varrho}_{R_0}$. In Proposition \ref{fact_not_optimal} below we will show that this is not the case for $\varrho\in (1/2,1]$.
\begin{proposition} \label{fact_not_optimal}
Let $\varrho\in\left(\frac12,1\right]$. Then
\begin{eqnarray}
    && e^{(p)}(l^{EE},F^{\varrho}_{R_0},V^1,\delta)=\Theta(n^{-1}+\delta),\\
    && e^{(p)}(l^{IE},F^{\varrho}_{\infty},V^2,\delta)=\Theta(n^{-1}+\delta),
\end{eqnarray}
as $n\to\infty$ and $\delta\to 0^+$.
%
\end{proposition}
\begin{proof}
Let $A=\min\{K,L\}$, $\eta=A$ and $f(t,y)=Ay$. Then $(\eta,f)\in F^\varrho_{\infty}$ and the exact solution to \eqref{eq:ode} is given by $z(t)=Ae^{A(t-a)}$,  $t\in [a,b]$. For each $n\in\mathbb{Z}_+$ such that $Ah\neq 1$ we obtain the following sequences of approximated values of function $z$ produced by explicit and implicit randomized Euler schemes, respectively, under exact information:
$$V^j = A(1+Ah)^j, \ \ \ U^j=A(1-Ah)^{-j} \ \ \ \text{for} \ \ \ j\in\{0,1,\ldots,n\},$$
where $h=\frac{b-a}{n}$. With the help of de L'H\^{o}pital's rule we get that $$\lim_{x\to\pm\infty}\left[ x\cdot \left( e^\gamma-\Bigl(1+\frac{\gamma}{x}\Bigr)^x \right) \right] = \frac{\gamma^2 e^\gamma}{2}$$
for all $\gamma>0$. Hence, 
$$\sup_{a\leq t \leq b} \left| z(t)-l^{EE}(t)\right| \geq |z(b)-V^n| = A \Bigl( e^{A(b-a)} - \Bigl(1+\frac{A(b-a)}{n}\Bigr)^n \Bigr) =  \Omega\Bigl(\frac{1}{n}\Bigr),$$
and
$$\sup_{a\leq t \leq b} \left| z(t)-l^{IE}(t)\right|\geq |z(b)-U^n| = A \Bigl( e^{A(b-a)} - \Bigl(1-\frac{A(b-a)}{n}\Bigr)^{-n} \Bigr) =  \Omega\Bigl(\frac{1}{n}\Bigr),$$
when $n\to\infty$. This combined with Theorem \ref{theorem_explicit} and Theorem \ref{theorem_implicit} completes the proof.
\end{proof}
 \begin{remark}
 In \cite{randRK} we have shown that the randomized Runge-Kutta scheme of order 2 is optimal in the class $F^{\varrho}_R$ for all $\varrho\in (0,1]$ and suitably chosen $R$. However, the class of corrupting function  functions $\tilde\delta$ considered in \cite{randRK} was smaller than $V^1$ -- corrupting functions were assumed to be bounded. We conjecture that the optimality of the randomized R-K method is preserved when the corrupting functions belong to the class $V^1$.
 \end{remark}
\section{Stability of randomized Euler schemes} \label{sec:stab}
In this section we investigate the stability issues of the randomized Euler schemes in the case of exact information, i.e. $\delta=0$. Typically the following test problem is used to analyze stability of numerical methods for ODEs:
\begin{equation}
	\label{TEST_PROBLEM}
		\left\{ \begin{array}{ll}
			z'(t)= \lambda z(t), \ t\geq 0, \\
			z(0) = \eta
		\end{array}\right.
\end{equation}
with $\lambda\in\mathbb{C}$, $\eta\neq 0$. The exact solution of \eqref{TEST_PROBLEM} is $z(t)=\eta\exp(\lambda t)$. Since in this problem the right-hand side function  $f(t,z) = \lambda z$  does not depend on the time variable, \eqref{TEST_PROBLEM} does not allow to capture randomization in Euler methods considered in this paper. Hence, we propose the following alternative test problem
\begin{equation}
	\label{TEST_PROBLEM2}
		\left\{ \begin{array}{ll}
			z'(t)= 2\lambda t z(t), \ t\geq 0, \\
			z(0) = \eta,
		\end{array}\right.
\end{equation}
with $\lambda\in\mathbb{C}$, $\eta\neq 0$. The exact solution of \eqref{TEST_PROBLEM2} is $z(t)=\eta\exp(\lambda t^2)$ and 
\begin{equation}
    \lim\limits_{t\to\infty} z(t)=0 \ \hbox{iff} \ \Re(\lambda)<0.
\end{equation}

Similarly as in \cite{randRK}, we consider three stability regions: 
\begin{eqnarray}
    &&\mathcal{R}^{MS}_{scheme} = \{h^2\lambda \in\mathbb{C} \colon W^k\to 0 \ \hbox{in $L^2(\Omega)$ as} \ k\to \infty \}, \notag\\
    &&\mathcal{R}^{AS}_{scheme} = \{h^2\lambda\in\mathbb{C} \colon W^k\to 0 \  \hbox{almost surely as} \ k\to \infty\},\notag\\
     &&\mathcal{R}^{SP}_{scheme} = \left\{ h^2\lambda\in \mathbb{C} \colon  W^k \to 0 \text{ in probability as} \ k\to \infty  \right\},
    \label{eq:regs}
\end{eqnarray}
where $(scheme, W)\in \{ (EE,V), (IE,U)\}$ and the sequences $\left(V^k\right)_{k=0}^\infty$ and $\left(U^k\right)_{k=0}^\infty$ are generated respectively by schemes \eqref{eq:explicit_euler} and \eqref{eq:implicit_euler} applied to the test problem \eqref{TEST_PROBLEM2} with $\delta=0$ (i.e. under exact information). Regions \eqref{eq:regs} are called the region of mean-square stability,  the region of asymptotic stability, and the region of stability in probability, respectively.

\begin{proposition} \label{prop_stab}
Stability regions of randomized Euler schemes have the following properties:

\begin{itemize}
\setlength{\itemsep}{0.2cm}
    \item[(i)] $\displaystyle \mathcal{R}_{EE}^{MS} =\mathcal{R}_{EE}^{AS} =\mathcal{R}_{EE}^{SP} = \emptyset$,
    \item[(ii)] $\displaystyle \mathbb{C}\setminus \left(\mathbb{R}_+\cup\{0\}\right) \subset \mathcal{R}^{MS}_{IE} \cap \mathcal{R}^{AS}_{IE} \cap \mathcal{R}^{SP}_{IE}$.
\end{itemize}
\end{proposition}

\begin{proof}
Let $\left(V^k\right)_{k=0}^\infty$ and $\left(U^k\right)_{k=0}^\infty$ be the sequences generated respectively by schemes \eqref{eq:explicit_euler} and \eqref{eq:implicit_euler} applied to the test problem \eqref{TEST_PROBLEM2} with $\delta=0$. In both cases $\theta_j = h(j-1+\tau_j)$ for $j\in\mathbb{Z}_+$.

For the explicit scheme, let us observe that
\begin{align*}
    V^k = V^{k-1} \cdot \left( 1+2\lambda h \theta_k \right) = \cdots = \eta \cdot \prod_{j=1}^k \left( 1+2\lambda h\theta_j \right)
\end{align*}
for $k\in\mathbb{Z}_+$. Moreover,
\begin{align*}
    \bigl|V^k\bigr|^2  & =  \left|\eta\right|^2 \cdot \prod_{j=1}^k \left| 1+2\lambda h\theta_j \right|^2 = \left|\eta\right|^2 \cdot \prod_{j=1}^k \left( 1+4\Re\left(\lambda\right)h\cdot\theta_j + 4\left|\lambda\right|^2h^2\cdot\theta_j^2 \right)  
\end{align*}
and 
\begin{equation} \label{eq:expl_ineq}
    1+4\Re\left(\lambda\right)h\cdot\theta_j + 4\left|\lambda\right|^2h^2\cdot\theta_j^2 \geq 1+4\Re\left(\lambda\right)h^2\cdot(j-1) + 4\left|\lambda\right|^2h^4\cdot(j-1)^2 > 2 
\end{equation}
with probability $1$ for all $h>0$, $\lambda\in\mathbb{C}\setminus\{0\}$ and sufficiently big $j\in\mathbb{Z}_+$. As a result, $V^k \to \infty$ as $k\to\infty$ for all convergence types defined in \eqref{eq:regs}. For completeness, let us note that for $h>0$ and $\lambda=0$ we have $\left|V^k\right| = \left|\eta\right| > 0$ for all $k\in\mathbb{Z}_+$. This leads to (i).

For the implicit scheme we have
\begin{equation*}
    U^k = U^{k-1} + h\cdot 2\lambda \theta_k U^k
\end{equation*}
and thus
\begin{equation*}
    U^k = \frac{1}{1-2h\lambda\theta_k}U^{k-1} = \cdots = \eta \cdot \prod_{j=1}^k \frac{1}{1-2h\lambda\theta_j}
\end{equation*}
for $k\in\mathbb{Z}_+$. Let us take $\lambda\in\mathbb{C}\setminus \left(\mathbb{R}_+\cup\{0\}\right)$ and $h>0$. Then we have $1-2\lambda h t \neq 0$ for all $t\geq 0$. Furthermore, 
 \begin{equation} \label{eq:impl_ineq}
    \left| \frac{1}{1-2\lambda h\theta_j} \right|^2 \leq \frac{1}{1-4\Re\left(\lambda\right)h^2(j-1)+4\left|\lambda\right|^2h^4(j-1)^2} <\frac12 
\end{equation}
with probability $1$ for sufficiently big $j\in\mathbb{Z}_+$. Hence, $$\bigl|U^k\bigr| =  \left|\eta\right|\cdot\prod_{j=1}^k \left| \frac{1}{1-2\lambda h\theta_j} \right| \longrightarrow 0 \ \text{ as }  \ k\to\infty$$ for $\lambda\in\mathbb{C}\setminus \left(\mathbb{R}_+\cup\{0\}\right)$,  $h>0$, and for all convergence types considered in \eqref{eq:regs}. This  completes the proof.
\end{proof}

\begin{remark} \label{rem_stab}
In order to obtain classical (deterministic) versions of Euler schemes, it suffices to set $\theta_j=t_{j-1}$ for all $j\in\mathbb{Z}_+\cup \{0\}$ (in case of explicit method) or $\theta_j=t_{j}$ for all $j\in\mathbb{Z}_+\cup \{0\}$ (in case of implicit method). Hence, inequalities \eqref{eq:expl_ineq} and \eqref{eq:impl_ineq} imply that $\displaystyle \mathcal{R}_{EE} = \emptyset$ and $\displaystyle \mathbb{C}\setminus \left(\mathbb{R}_+\cup\{0\}\right) \subset \mathcal{R}_{IE}$, where $\mathcal{R}_{EE}$ and $\mathcal{R}_{IE}$ are the absolute stability regions for the deterministic explicit and implicit Euler schemes. This leads to the conclusion that randomization has no impact on the stability of Euler methods for both test problems \eqref{TEST_PROBLEM} and \eqref{TEST_PROBLEM2}, although for the latter the right-hand side function depends on the time variable (which is randomized). Both problems show however significant advantage of the implicit Euler scheme over the explicit one in terms of stability. This is particularly apparent for \eqref{TEST_PROBLEM2}, where stability regions of explicit and implicit methods represent two extreme cases. 
\end{remark}

\section{Conclusions}

We have established error bounds for randomized Euler schemes under mild assumptions about the right-hand side function. Our analysis has been performed in the setting of inexact information. It turns out that the randomized Euler schemes are optimal only for values of H\"older exponent $\varrho$ not greater than $\frac12$. Results for the explicit randomized Euler scheme are proven for broader classes of the right-hand side functions and noise functions than analogous results for the implicit scheme. On the other hand, the implicit scheme turned out to have a significant advantage in terms of stability. 
\section*{Appendix}
The following lemma can be proven in the same fashion as Lemma 1(i) in \cite{randRK}.
\begin{lemma} \label{lm:sol}
Let $(\eta,f)\in F^\varrho_{R_2}$, where \begin{equation}
    \label{eq:R2} R_2 = K(1+b-a)e^{K(b-a)} + K.
\end{equation}
Then 
\begin{itemize}
    \item[(i)] equation \eqref{eq:ode} has a unique solution $z=z(\eta,f)$ such that  $z\in\mathcal{C}^1\left([a,b]\times\mathbb{R}^d\right)$ and $z(t) \in B(\eta,R_2)$ for all $t\in[a,b]$;
    \item[(ii)] there exist  $C_1 = C_1(a,b,K)\in (0,\infty)$ and $C_2 = C_2(a,b,K,L)\in (0,\infty)$ such that for all $t \in [a,b]$
\begin{equation}\label{eq:diff_z}
 \|z(t)-z(s)\| \leq C_1 |t-s|, 
\end{equation}
\begin{equation} \label{eq:5}
\left\|z'(t)-z'(s)\right\| \leq C_2|t-s|^\varrho.
\end{equation}
\end{itemize}
\end{lemma}

Next lemma is used to show existence and uniqueness of a measurable solution to the implicit randomized Euler scheme. Its proof can be found in \cite{backward_euler} (Lemma 4.3).
\begin{lemma} \label{lemma:measurable}
Let $\tilde{\mathcal{F}}$ be a complete sub $\sigma$-algebra of the $\sigma$-algebra $\Sigma$, $M \in \tilde{\mathcal{F}}$ with $\mathbb{P}\left(M\right) = 1$ and $h\colon \Omega\times \mathbb{R}^d \to \mathbb{R}^d$ such that the following conditions are fulfilled.
\begin{itemize}
\item[(i)] The mapping $x\mapsto h(\omega,x)$ is continuous for every $\omega\in M$.
\item[(ii)] The mapping $\omega\mapsto h(\omega,x)$ is $\tilde{\mathcal{F}}$-measurable for every $x\in\mathbb{R}^d$.
\item[(iii)] For every $\omega\in M$ there exists a unique root of the function $h(\omega,\cdot)$.
\end{itemize}
Define the mapping $$Q\colon \Omega \ni \omega \mapsto Q(\omega) \in \mathbb{R}^d,$$ where $Q(\omega)$ is the unique root of $h(\omega,\cdot)$ for $\omega\in M$ and $Q(\omega)$ is arbitrary for $\omega\in\Omega\setminus M$. \medskip

Then $Q$ is $\tilde{\mathcal{F}}$-measurable.
\end{lemma}
 {\noindent\bf Acknowledgments}
\newline
This research was partly supported by the National Science Centre, Poland, under project 2017/25/B/ST1/00945.

\end{document}